\documentclass[11pt]{article}
\usepackage[english]{babel}
\usepackage{amsthm, amsmath, amssymb, amsbsy, amsfonts, latexsym, euscript}
\usepackage{graphicx}
\newtheorem{theorem}{\bf Theorem}[section]

\newtheorem{lemma}{\bf Lemma}[section]
\newtheorem{corollary}{\bf Corollary}[section]

\theoremstyle{remark}

\newtheorem{definition}{\bf Definition}[section]

\newcommand{\grad}{\textrm{grad}\,}

\newcommand{\h}{\textrm{H}}
\newcommand{\V}{\textrm{Vol}}

\begin{document}
\title{Geometry of generalized
virtual polyhedra}

\author{Askold Khovanskii \thanks{The work was partially supported by the Canadian Grant No. 156833-17. }}


\maketitle

\begin{abstract}{Partial generalizations of virtual polyhedra theory (sometimes under different names)  appeared recently   in the theory of torus manifolds. These generalizations look very different from the original virtual polyhedra theory.  They are based on simple arguments from homotopy theory while the original theory is based on integration over Euler characteristic. In the paper we explain how these generalizations are related to the classical theory of convex bodies and to the original  virtual polyhedra theory. The paper basically contains no proofs: all proofs and all details can be found in the cited literature.
The paper is based on my talk  dedicated to V. I. Arnold's 85-th anniversary at the International Conference on Differential Equations and Dynamical Systems 2022
(Suzdal).}\end{abstract}

\section{Introduction. Virtual convex polyhedra and their polynomial measures}

Convex polyhedra in the linear space $\mathbb R^n$ form a convex cone in the following way. One can multiply a convex polyhedron $\Delta$   by any nonnegative real number $\lambda$  (i.e. take its  dilatation $\lambda \Delta$  centred at the origin with the factor $\lambda$) and add two convex polyhedra $ \Delta_1, \Delta_2$  in Minkowski sense. Recall that the Minkowski sum of $\Delta_1, \Delta_2\subset \mathbb R^n$ is the set $\Delta$ of the points $z$ representable in the form $z=x+y$, where $x\in \Delta_1, y\in \Delta_2$.

A convex chain is a function on $\mathbb R^n$ representable as a finite linear combination with real coefficients of characteristic functions of closed convex polyhedra (of different dimensions).

Convex chains form a real vector space in a natural way. One can further define a product $f*g$ of two chains $f,g$ as follows. If $f$ and $g$ are characteristic functions of closed  convex polyhedra $\Delta_1, \Delta_2\subset \mathbb R^n$ then, by definition, the chain  $f*g$ is the characteristic function of $\Delta=\Delta_1+\Delta_2$ (where the addition is understood in the Minkowski sense). This product can be extended by linearity to the space of convex chains.

Note that it is not obvious at all that the above product is well defined. Indeed, convex chain can be represented as a linear combination of characteristic functions in many different ways, and independents of product $f*g$ of such representations of $f$ and $g$ is not obvious. One can prove \cite{1} that the product is well defined using as the tool integration over Euler characteristic \cite{2}.

Convex chains in $\mathbb R^n$  with the multiplication $*$ form a real algebra with the identity element~$\mathbf 1$, which is the characteristic function of the origin in $\mathbb R^n$. The characteristic function $\chi_\Delta$ of a closed convex polyhedron $\Delta\subset \mathbb R^n$ are invertible in the algebra of convex chains. More precisely, the following theorem holds.

\begin{theorem} Let $\Delta\subset \mathbb R^n$ be a convex polyhedron and let $-\Delta_0$ be the set of interior points  (in the intrinsic  topology of $\Delta$) of the polyhedron $-\Delta$ symmetric to $\Delta$ with respect to the origin. Then
$$
(-1)^{\dim \Delta}\chi_{-\Delta_0} * \chi_\Delta =\mathbf 1.
$$
In other words, the convex chain $(-1)^{\dim \Delta}\chi_{-\Delta_0}$ is inverse to $\Delta$ with respect to the addition in Minkowski sense (extended to the space of convex chains).
\end{theorem}

Algebra of convex chains contains the multiplicative subgroup generated by characteristic functions of closed convex polyhedra. Elements of that group are called virtual polyhedra in $\mathbb R^n$.

Let us fix closed convex polyhedra $\Delta_1,\dots, \Delta_k\subset \mathbb R^n$.  For any  $k$-tuple of  nonnegative integral numbers  $\mathbf n=( n_1,\dots,n_k)$  one can defined the polyhedron $\Delta(\mathbf n)=\sum n_i\Delta_i$.

The following sentence can be considered as a slogan of virtual polyhedra theory: ``A natural  continuation of the function $\Delta(\mathbf n)$ (whose values are convex polyhedra) to $k$-tuples $\mathbf n=(n_1,\dots,n_k)$ of integral numbers (some of which could be negative) is a convex chain $\tilde \Delta(\mathbf n)$ defined by the following formula

$$ \tilde \Delta(\mathbf n)=\chi^{n_1}_{\Delta_1}* \dots * \chi^{n_k}_{\Delta_k}.$$

This slogan has a following justification: value of a polynomial measures (see an example of such measure below) on a chain $\tilde \Delta(\mathbf n)$ is a polynomial of $\mathbf n$. Generalizations of virtual polyhedra theory suggest other families of cycles depending on parameters, such that integrals against such cycles of a differential form with polynomial coefficients are polynomial in parameters.

Let us present an example of a polynomial measure on convex polyhedra with integral vertices and a justification of the slogan of virtual polyhedra theory. Let $P:\mathbb R^n\to \mathbb R$ be a polynomial of degree $m$.
With $P$ one can associate the following measure $\mu$ on convex polyhedra $\Delta$ with integral vertices: $\mu(\Delta)=\sum_{x\in \mathbb Z^n\cap \Delta}P(x)$. One can prove that the function $\mu(\Delta(\mathbf n)$ is a degree $\leq (n+m)$ polynomial on $k$-tuples $\mathbf n$ of nonnegative  integral  numbers.

The following Theorem justifies the slogan of virtual polyhedra theory.

\begin{theorem} Let $P$ be a polynomial of degree $m$ and let $\tilde F(\mathbf n)$ be the function on $k$-tuples $\mathbf n=(n_1,\dots,n_k)$ of integral numbers (which could be negative) defined by the formula
$$
\tilde F(\mathbf n)=\sum _{x\in Z^n} \chi_{\Delta_1}^{n_1}(x)*\dots *\chi_{\Delta_k}^{n_k}(x)P(x).
$$
Then $\tilde F(\mathbf n)$  is a degree $\leq (n+m)$ polynomial on $k$-tuple $\mathbf n$ which coincides with $F(\mathbf n)$ on $k$-tuples with nonnegative components.
\end{theorem}

Virtual polyhedra theory allows to develop a general theory of polynomial finite additive measures on convex polyhedra (see \cite{1}), which contains wide generalizations of the above theorem.

The virtual polyhedra  theory  was motivated by cohomology theory of complete toric varieties, with coefficients in sheafs invariant under the torus action. In particular it provides a combinatorial version of Riemann--Roch theorem for such varieties \cite{3}, which also could be considered as  a multidimensional version of the classical Euler--MacLuren formula (see \cite{3}).

The general theory is applicable  to singular  polynomial measures on polyhedra (such as the measure, which associates to a polyhedron the number of integral points in  it) which could take nonzero value on polyhedra $\Delta$ with $\dim \Delta<n$. However, if one is interested in nonsingular polynomial measures, which vanish on polyhedra whose dimension is smaller than $n$, one can totally neglect all polyhedra of dimension $<n$ in convex chains. This leads to a significant simplification  of virtual polyhedra theory, which captures smooth polynomial measures  (an which is not appropriate for studying   singular measures).

 Simplified  theory  is still  useful. In particular, it allows to provide a topological proof of Bernstein--Koushnirenko--Khovanskii (BKK) theorem. More generally, using a description of algebras with Poincare duality (see for example \cite[Section 6]{4} or \cite{5}) it allows to describe the cohomology ring  $H^*(M,\mathbb Z)$ of a smooth complete toric variety $M$ in terms of volume function on virtual integral convex polyhedra (so-called Khovanskii--Pukhlikov description of the ring $H^*(M,\mathbb Z)$).

 In this paper we only deal with simplified versions of the virtual polyhedra theory which deal only with nonsingular measures  as well as its generalizations. We also mention some topological applications of these generalizations. We start with geometric meaning of a virtual convex body and its volume for the difference of two strictly convex bodies with smooth boundaries. We also will present some applications of mixed volume and virtual polyhedra in algebra.

\section{Virtual strictly convex bodies and their volumes}

Formal virtual convex body is a formal difference of compact convex bodies (which in general are not polyhedra).

Similar to polyhedra, compact convex bodies in $\mathbb R^n$ form a convex cone with respect to Minkowski addition and dilation with positive factors centered at the origin. Moreover, the addition of convex bodies satisfies the cancelation property, i.e. if for a convex body $\Delta$ the identity $\Delta_1+\Delta=\Delta_2+\Delta$ implies that $\Delta_1=\Delta_2$. Hence one can generate a group by formal differences of convex bodies with
$\Delta_1-\Delta_2=\Delta_3-\Delta_4$ whenever $\Delta_1+\Delta_4=\Delta_3+\Delta_2$.

By Minkovsky's Theorem, the volume is a homogeneous degree $n$ polynomial on the cone of convex bodies. More concretely,
if $\Delta_1, \Delta_2$ are convex bodies,  $\lambda,\mu\geq0$
then the volume $\V(\lambda\Delta_1+\mu \Delta_2)$ is a homogeneous degree $n$ polynomial in $(\lambda,\mu)$. Therefore, the volume can be extended to the linear space of formal differences of convex bodies as a homogeneous degree $n$ polynomial. In Section~\ref{sec:virtgeom} we give a geometric interpretation of virtual convex bodies as well as their volumes.

Since the volume is a homogeneous polynomial of degree $n$ on the cone of convex bodies in $\mathbb R^n$, it admits a polarization $\V(\Delta_1,\dots,\Delta_n)$. That is $\V(\Delta_1,\dots,\Delta_n)$ is a unique function of $n$-tuple of convex bodies $\Delta_1,\dots,\Delta_n$ with the following properties:
\begin{itemize}
\item[1.] $\V(\Delta_1,\dots,\Delta_n)$ is linear in each argument, with respect to Minkowski addition;
\item[2.] $\V(\Delta_1,\dots,\Delta_n)$ is symmetric;
\item[3.] on a diagonal it is equal to the volume, i.e. $\V(\Delta,\dots, \Delta)=\V(\Delta)$.
\end{itemize}
The polarization of the volume polynomial is called the mixed volume. By multi-linearity mixed volume can be extended to $n$-tuples of virtual convex bodies.

\section{Volume and mixed volume in algebra}
In  this section we briefly recall the relation of mixed volumes of virtual polytopes with algebraic geometry. Let $\Delta_1,\dots,\Delta_n$ be a collection of convex polyhedra with integral vertices.

The following question was originated by Vladimir Igorevich Arnold in the middle 1970-th:
``Let $P_1,\dots,Pn$ be a generic $n$-tuple of Laurent polynomials with given Newton polyhedra $\Delta(P_i)=\Delta_i$. How many roots  does a system of  equations $P_1=\dots=P_n=0$ have in $(\mathbb C^*)^n$?''

The answer is given by the Bernstein-Koushnirenko-Khovanskii (BKK) theorem which was originaly proved by A.G.~Koushnirenko and D.N.~Bernstein. In later work, I found many generalizations and different proofs  of that result.

\begin{theorem}[BKK theorem]
The number of solutions is equal to $n!\V(\Delta_1,\dots,\Delta_n)$.
\end{theorem}

One generalization of BKK theorem comes if we consider rational functions on $(\mathbb C^*)^n$ instead of Laurent polynomials. Let $\frac{P_1}{Q_1},\dots, \frac{P_n}{Q_n}$ be a generic $n$-tuple of rational functions with given Newton polyhedra $\Delta(P_i)=\Delta_i$ and $\Delta(Q_i)=\Delta_i'$. Then the intersection number in $(\mathbb C^*)^n$ of the principal divisors of these rational functions is equal to multiplied by $n!$ mixed volume of the virtual polyhedra $\tilde \Delta_i=\Delta_i-\Delta_i'$, i.e. is equal to
$n!\V(\tilde \Delta_1,\dots,\tilde \Delta_n)$ (see for details \cite{6}).

\section{Geometric meaning of Virtual strictly convex bodies}\label{sec:virtgeom}
First recall that the support function  $\h_\Delta$ of a compact convex body $\Delta\subset \mathbb R^n$ is a function on the dual space $(\mathbb R^n)^*$ defined by the following formula:
$$
\h_\Delta(\xi)=\max_{x\in \Delta} \langle \xi, x\rangle.
$$
One can further associate a support function to a virtual convex body. Indeed, the support function depend linearly on the convex body, thus it  can be naturally extended to differences of convex bodies:
$\h_{\Delta_1-\Delta_2}=\h_{\Delta_1}-\h_{\Delta_2}$. The support function $\h_\Delta$ of a (vitual) convex body $\Delta$ is a degree one homogeneous function. More precisely, for $\lambda\geq 0$ the following relation holds:
$\h_\Delta(\lambda \xi)=\lambda \h_\Delta(\xi)$.

In what follows, we assume that in $\mathbb R^n$ an Euclidian metric is fixed, which allows to identify $(\mathbb R^n)^*$ with $\mathbb R^n$. Assume further that $\Delta$  has smooth boundary and it is strictly convex. Then for $\xi$ not equal to zero the inner product $\langle \xi,x\rangle$ attaints maxima at one point $a$ of $\partial \Delta$ only and this point $a (\xi)$ is equal to
$\grad \h_{\Delta}(\xi)$.

\begin{lemma}\label{lemma4.1} The vector-function $\grad  \h_{\Delta}(\xi)$ restricted to the unite sphere $S^{n-1}$ defines a map from $S^{n-1}$ to the boundary $\partial \Delta$ of the strictly convex body $\Delta$. Moreover,
this map is inverse to the Gauss map $g:\partial \Delta\to S^{n-1}$.
\end{lemma}

To a virtual convex body $\Delta$ with a smooth on $\mathbb R^n\setminus \{0\}$ support function $\h_{\Delta}$ One can associate the image $\grad \h_{\Delta (S^{n-1})}$ of the unite sphere under the map $\grad \h_\Delta:S^{n-1}\to \mathbb R^n$ This image has a natural parametrization by the sphere $S^{n-1}$. The correspondence $\Delta \to \grad \h_{\Delta}$ provides a map from the space of virtual convex bodies with smooth support function $\h_{\Delta}$ to the linear space of gradient mappings from $S^{n-1}$ to $\mathbb R^n$.

Consider $(n-1)$-form
$\omega =x_1d x_2\wedge \dots\wedge dx_n$ on $\mathbb R^n$. Notice that the differential $d\omega$ is the standard volume form on $\mathbb R^n$. The following statement is a direct corollary of Lemma~\ref{lemma4.1} and  Stock's formula.

\begin{corollary} \label{cor4.1}The volume of a convex body $\Delta$ with smooth strictly convex boundary $\partial \Delta$ is equal to
$\int_{S^{n-1}}f^*\omega$
where $f= \grad  \h_{\Delta}$ restricted to the sphere $S^{n-1}$.
\end{corollary}

Corollary~\ref{cor4.1} provides a proof of Minkowski Theorem for convex bodies with smooth strictly convex boundaries. Indeed,
$$
(\grad  \h_{\lambda \Delta_1 +\mu \Delta_2})^*\omega=(\lambda \cdot \grad  \h_{\Delta_1}+\mu\cdot \grad \h_{\Delta_2})^*\omega
$$
is an $(n-1)$-form whose coefficients are degree $n$ homogeneous polynomials in $(\lambda,\mu)$. Moreover, since the above formula for the volume is written in terms of support functions, it is applicable to virtual convex bodies. More concretely,  for a virtual convex body $\Delta=\Delta_1-\Delta_2$ with $\Delta_1,\Delta_2$ strictly convex bodies with smooth boundaries, let $f$ be $\grad \h_\Delta= \grad (\h_{\Delta_1}-\h_{\Delta_2})$ restricted to the unite sphere. Then one has $\V(\Delta)=\int_{S^{n-1}}f^*\omega$.

Now we will give a different presentation for the volume of virtual convex bodies which is applicable to the case of generalized virtual polyhedra. Let $f: S^{n-1}\to \mathbb R^n$ be a smooth mapping of the unite sphere to $\mathbb R^n$. The image $f(S^{n-1})$ of the unit sphere $S^{n-1}$  cuts the space $\mathbb R^n$ into a collection of connected open  bodies.

\begin{definition} A {\sl winding number} $W_f(U) $ where $U$ is an open connected component of $\mathbb R^n\setminus f(S^{n-1})$ is a mapping degree of the map $\tau_a:S^{n-1}\to S^{n-1}$ where $\tau (\xi)=\frac{f(\xi)-a}{|f(\xi)-a|}$ for $\xi)\in S^{n-1}$ and $a$ is any point in $U$.
\end{definition}

Informally, the number $W_f(U)$ shows how many times the image $f(S^{n-1})$ of the sphere $S^{n-1}$ is rotating around~$U$.

\begin{definition} Let $\h (\xi)$ be a smooth function on $\mathbb R^n\setminus \{0\}$ which is homogeneous of degree one. Then the {\sl virtual convex body with the support function} $\h$ is defined as a chain
$$
\sum_U W_f(U) U,
$$
where $f=\grad \h$ and the sum is taken over all  bounded connected components of the complement $\mathbb R^n\setminus f(S^{n-1})$.
\end{definition}

\begin{theorem}\label{the4.1}  The volume of the virtual convex body with a smooth support function $\h$ on $\mathbb R^n\setminus \{0\}$ is equal to the integral of the volume form against the chain $\sum W_f(U) U$ associated with the virtual convex body. In other words, the volume of virtual bogy is equal to
$$\sum_U W_f(U)\V(U)$$
where  $\V(U)$ is the volume of $U$.
\end{theorem}

The proof follows from the formula for the volume of virtual convex body and from Stock's formula. Theorem~\ref{the4.1} has the following automatic generalization:

\begin{theorem} An integral of degree $m$ polynomial $P$ over virtual convex  body  with a smooth on $\mathbb R^n\setminus \{0\}$  support function $\h$ is equal to the integral of the polynomial $P$ against the chain, associated with this virtual  convex body, i.e. is equal to
$$
\sum_U W_f(U)\int_U Pdx_1\wedge\dots \wedge dx_n.
$$
\end{theorem}

\begin{proof} Theorem can be proven in the same way as Theorem~\ref{the4.1}. It is enough to replace  the form $\omega=x_1 d x_2\wedge \dots \wedge d x_n$ with  a form
$Qdx_2\wedge \dots \wedge d x_n$ such that $Q$ is degree $m+1$ polynomial satisfying $\partial Q/ \partial x_1=P$.
\end{proof}

One can generalize above theorems in the following directions:

1) Instead of the unite sphere $S^{n-1}$ and its gradient mappings to $\mathbb R^n$ one can take any piecewise smooth $(n-1)$-cycle $\Gamma$ and consider  the space of piecewise smooth mappings $f:\Gamma\to \mathbb R^n.$ An integrals against $\Gamma$ of the form $f^*\omega$ where $\omega$ is a fixed $(n-1)$ form with polynomial coefficients  on $\mathbb R^n$  is a polynomial on the space of maps $f$ from $\Gamma$ to $\mathbb R^n$. The same polynomial on the space of mapping $f$ one can obtain by integrating $n$-form $d\omega$ against the chain $\sum W_f(U) U$, where $U$ are connected components of $\mathbb R^n\setminus f(\Gamma)$ and $W_f(U)$ is the mapping degree of map $\tau:\Gamma\to S^{n-1}$ where $\tau(x)=\frac{f(x)-a}{|f(x)-a|}\in S^{n-1}$ where $x\in \Gamma$ and $a$ is any point in $U$. The chain $W_f(U)U$ is an analog  of the chain associated with a  virtual convex body.

2) Let $\Gamma$ be $(n-1)$-cycle as above, and let $M(\Gamma,L)$ be the space of piecewise linear mapping of $\Gamma$ to a real linear space $L$. With a fixed $(n-1)$ form $\omega$ with polynomial coefficients  on the space $L$ one can associate the polynomial function on $M(\Gamma,L)$ whose value on $f\in M(\Gamma, L)$ is equal to $\int_\Gamma f^*\omega$. In such generalization one has integrals which depend in polynomial way on parameters (but in such generalization there are no chains analogous to chains associated with virtual convex polyhedra).

\section{Analogues virtual polyhedra and their volume}

Let us return to the original definition of virtual polyhedra.  With any given convex polyhedron $\Delta_0$ one associates a subgroup of virtual  polyhedra analogues to $\Delta_0$. In this section we first recall this construction and then describe a simplified theory of virtual polyhedra.

First recall that each convex  polyhedron $\Delta$ defines a dual fan $\Delta^\bot$ in the following way.  Two covectors are said to be $\Delta$-equivalent if they attaint maxima at the same face of $\Delta$.  A set of all $\Delta$-equivalent covectors form a cone (which is open in intrinsic topology).  Closures of such cones form the dual fan $\Delta^\bot$ for $\Delta$.

\begin{definition}
Two polyhedra $\Delta_1,\Delta_2$ are called analogous if their dual fans coincide. In particular, for each face of $\Delta_1$ there is exactly one face of $\Delta_2$ parallel to it.
\end{definition}

The following lemma is straightforward to show.
\begin{lemma}
Let $\Delta_1,\Delta_2$ be convex polyhedra analogues to $\Delta_0$. Then $\Delta_1+\Delta_2$ is also analogues~$\Delta_0$.
\end{lemma}

If a virtual polyhedron $\Delta$ is representable as a difference $\Delta_1-\Delta_2$ of polyhedra analogues to $\Delta_0$ then we will say that the virtual polyhedron $\Delta$ is analogues to $\Delta_0$. In other words, a virtual polyhedron analogues to $\Delta_0$ if the corresponding convex chain is representable in the form $\chi_{\Delta_1}*\chi_{\Delta_2}^{-1}$ where $\chi_{\Delta_i}$ is the characteristic function of $\Delta_i$.  Note that a virtual polyhedra $\Delta_1-\Delta_2$ depends on its support function (and is independent of the representation this function in the form $\h_{\Delta_1}-\h_{\Delta_2}$).

\subsection{Simplified version of analogous convex polyhedra theory}

If one  is only interested with nonsingular measures of a virtual polyhedron, one can neglect polyhedra of dimension $<n$ in the convex chain associated with a virtual polyhedron analogues to $\Delta_0$.  This leads to a simplified theory of virtual polyhedra which can be described  using support functions in a way, similar to the description of virtual convex bodies with smooth boundary presented above.

Convex polyhedra are not strictly convex and the Gauss map from the unite sphere to the boundary of convex polyhedra is not defined.   But one can defined (up to a homotopy) an analog of Gauss map from one polyhedron to an analogues  polyhedron.

\begin{figure}[htbp]
\begin{center}
\includegraphics[scale=.3]{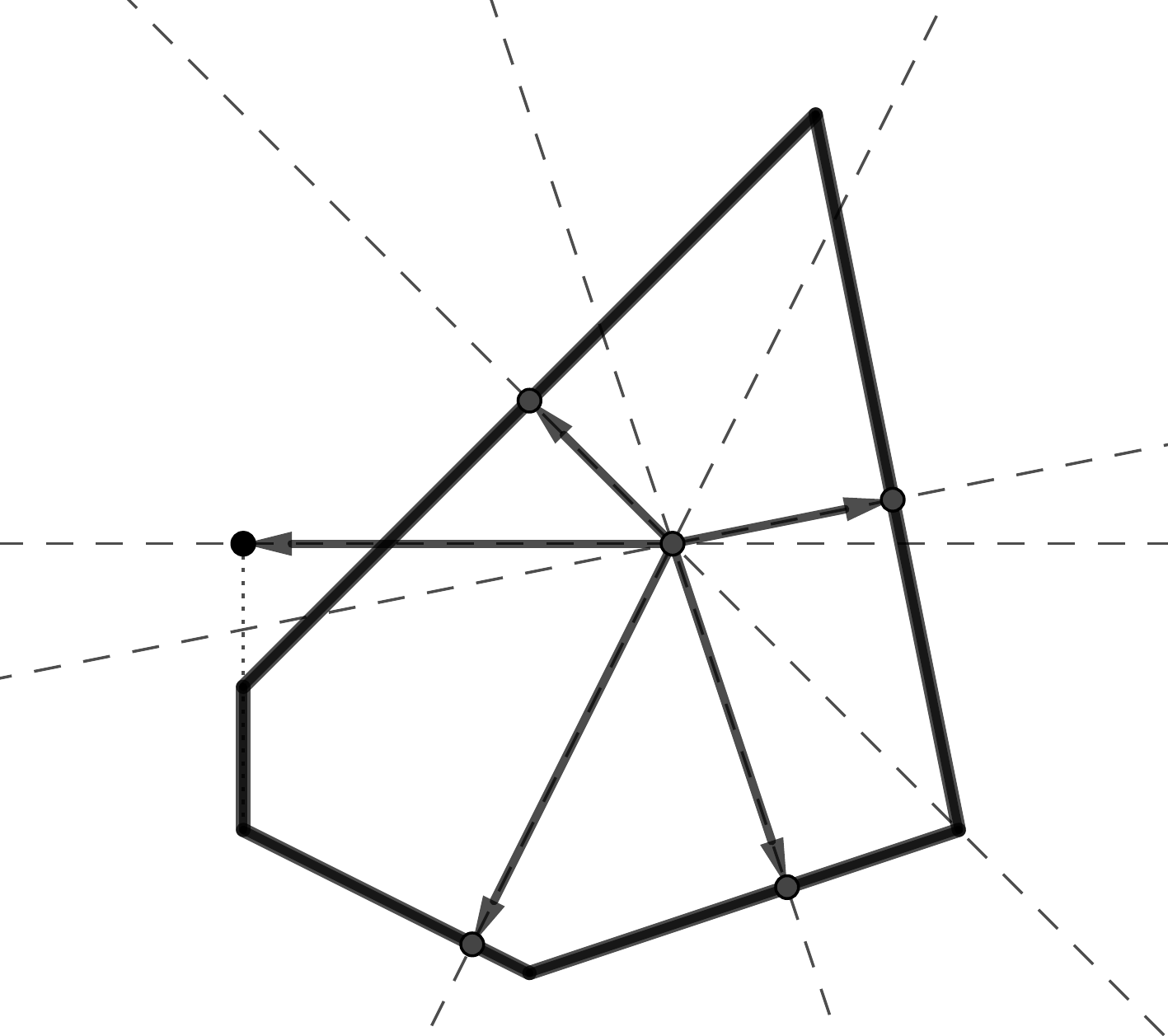}
\end{center}
\caption{Dual fan to a convex $5$-gon  }
\label{fig:fan}
\end{figure}

\begin{figure}[htbp]
\begin{center}
\includegraphics[scale=0.2]{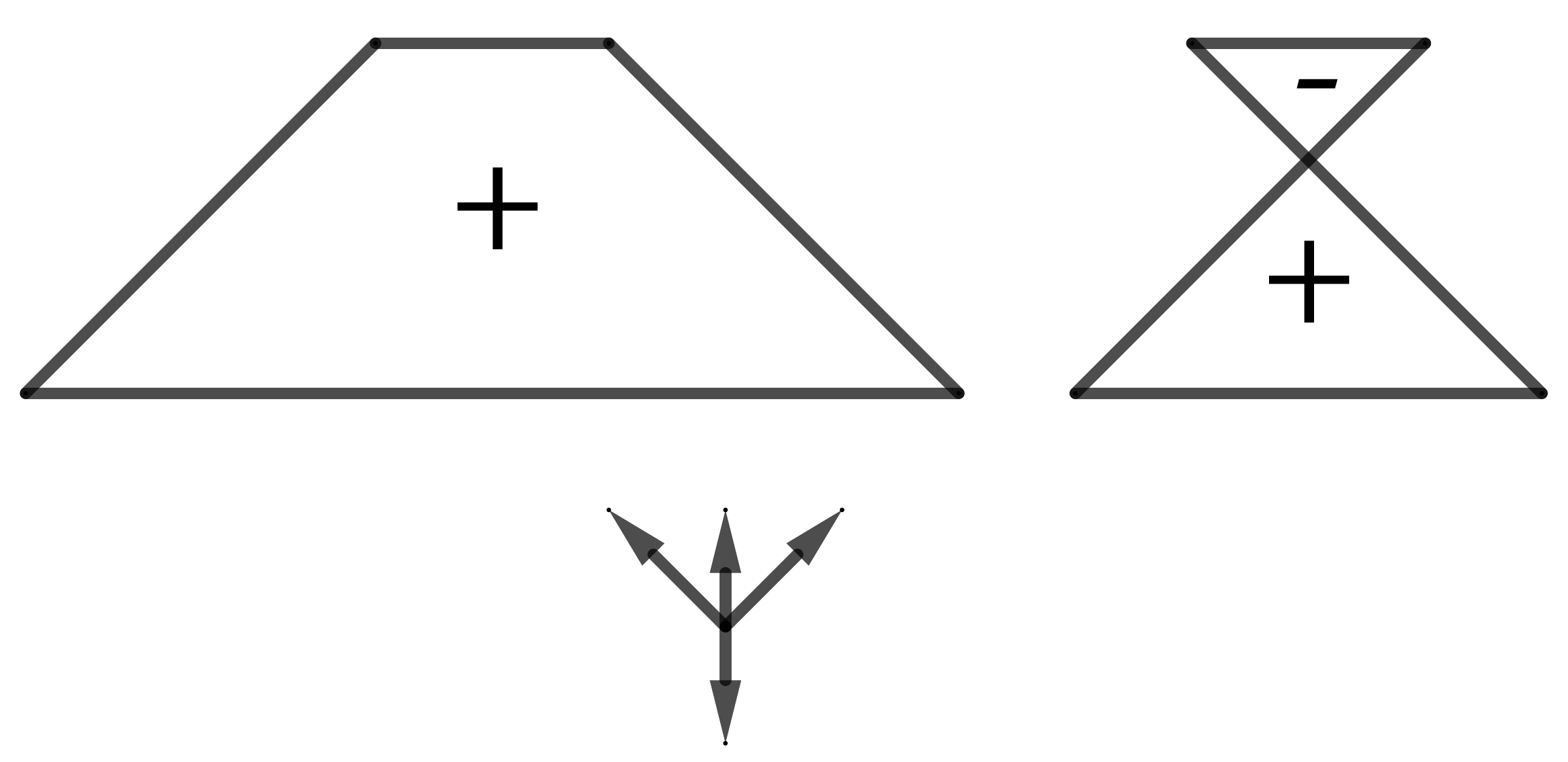}
\end{center}
\caption{Trapezoid, its dual fan and a virtual $4$-gon analogous to it}
\label{fig:trap}
\end{figure}

Let us fix a convex polyhedron $\Delta_0$. In what follows, it will play a role of the unite sphere in our construction.

To each polyhedron $\Delta$ analogous to $\Delta_0$ we associate the union $L_{\Delta}$ of affine hyperplanes $L_{\Gamma_i}$ which are affine spans of the facets $\Gamma_i$ of $\Delta$ (i.e. faces of $\Delta$ having dimension $(n-1)$).

\begin{definition} A continuous map $f_\Delta :\partial \Delta_0\to L_{\Delta}$ is a Gauss type map if the following condition holds: if $x\in \partial \Delta_0$ belongs to the closure of a $(n-1)$-dimensional face $\Gamma_i^0$ of $\Delta_0$ then  $f(x)$ has to belong to $L_{\Gamma_i}$ where $\Gamma_i$ and $\Gamma_i^0$ are parallel faces of $\Delta$ and $\Delta_0$.
\end{definition}

\begin{lemma}\label{lem5.2}
\begin{enumerate}
\item For any $\Delta$ analogous to $\Delta_0$ there exists a piecewise smooth Gauss type map~$f_\Delta$.

\item  Moreover, $f_\Delta$ can be defined in such a way that it depends linearly on $\Delta$, i.e. $f_{\lambda \Delta_1+\mu \Delta_2}=\lambda f_{\Delta_1}+\mu f_{\Delta_2}$.

\item Any two Gauss type maps from $\partial \Delta_0$ to $L_\Delta$ are homotopy equivalent to each other.
\end{enumerate}
\end{lemma}

Now we are ready to define the volume of a virtual polyhedron and  the integral of a polynomial form over a virtual polyhedron.

First let us associate a collection of affine hyperplanes to a virtual polyhedron $\Delta$ which is analogous to $\Delta_0$. Let $\h$ be a support function of $\Delta$. Then $\h$ is  a piecewise linear function on $\mathbb R^n$ which is linear at each cone of the dual fan $\Delta_0^\bot$ of $\Delta_0$. Then $\h$ defines a collection $L(\h)$ of hyperplanes $L_{\Gamma_i(\h)}$ parallel to the facets $\Gamma_i$ of $\Delta_0$ in the following way.

\begin{definition} Let $e_i$ be a normal vector, orthogonal to $\Gamma_i$ normalized in a certain way (say, has unite length, or is an irreducible integral vector). Then the hyperplane $L_{\Gamma_i(\h)}\subset \mathbb R^n$ is given by the equation
$$
\langle x,e_i\rangle=\h(e_i).
$$
\end{definition}

It is easy to check the following lemma.

\begin{lemma} If $\bigcap \Gamma_{i_j}=F$ is a non-empty face of $\Delta_0$, then $\bigcap L_{\Gamma_{i_j}(\h)}$ is an affine space parallel to~$F$.
\end{lemma}

\begin{definition} A Gauss type map $f_{\h}$ for a virtual polyhedron with the support function $\h$ is a map $f_{\h}: \partial \Delta_0 \to L(\h)$ which maps a face $F=\cap \Gamma_{i_j}$ of $\Delta_0$ to the affine space $L_{\h(F)}=
\cap L_{\Gamma_{i_j}(\h)}$.
\end{definition}

The statement of Lemma~\ref{lem5.2} also holds for virtual polyhedra. More precisely one gets the following lemma.

\begin{lemma}
\begin{enumerate}
\item   For any $\h$ which is linear at each cone of  $\Delta_0^\bot$ there exists a Gauss type map~$f_\h$.

\item Moreover, $f_\h$ can be defined in such a way that it will depends linearly on $\h$, i.e. $f_{\lambda \h_1+\mu \h_2}=\lambda f_{\h_1}+\mu f_{\h_2}.$

\item  Any two Gauss type maps from $\partial \Delta_0$ to $L(\h)$ are homotopy equivalent to each other.
    \end{enumerate}
\end{lemma}

\begin{definition} A winding number $W_{f_\h}(U) $ where $U$ is an open connected component of $\mathbb R^n\setminus L(\h)$, is the mapping degree of a map $\tau:\partial \Delta_0\to S^{n-1}$ where $\tau (x)=\frac {f_\h(\xi)-a}{||f_\h(\xi)-a||}$ for $x\in \partial \Delta_0$ and $a$ is any point in $U$.
\end{definition}

Analogous to the case of virtual convex bodies with smooth boundary, to the virtual polyhedron with support function $\h$ we associate the chain
$$
\sum_U W_{f_{\h}}(U)
$$
where the sum is taken over open connected components of $\mathbb R^n\setminus L(\h)$. One can prove the following theorem.

\begin{theorem}\label{thm5.1}The chain $\sum W_{f_{\h}}(U) U$ can be obtained  from the virtual polyhedra with the support function $\h$ by neglecting all polyhedra in the chain whose dimension is smaller than $n$.
\end{theorem}

Thus an integral over the virtual polyhedron of any $n$-form with polynomial coefficients can be obtain by integrating this form over the chain $\sum W_{f_{\h}}(U) U$. One can deal with integrals of such type using simple arguments which we applied above to similar integrals over virtual convex bodies with smooth boundaries (and the integration over Euler characteristic technique  is not needed here).

\begin{theorem}\label{thm5.2} Values of
$$\int_{\partial \Delta_0}f_\Delta^* (\omega) \quad \text{and} \quad \int_{\partial \Delta_0} f_\Delta^* (Qdx_2\wedge \dots\wedge d x_n)$$
are equal to
$$\sum W_{f_{\h}}(U)\int_U d x_1\wedge \dots\wedge d x_n\; \text{and} \;\sum W_{f_{\h}}(U)\int_U d x_1\wedge \dots \wedge d x_n$$
correspondingly.
\end{theorem}

For a convex support function $\h$ linear on each cone of $\Delta^\bot$ one associates an oriented polyhedron $\Delta(\h)$ with the support function $\h$. If one is interested in integrals against a chain of polynomial differential forms then  the natural continuation of the functor $\h\to \Delta(\h)$ to non-convex support functions linear on each cone of $\Delta^\bot$ is the functor $\h \to \sum W_{f_{\h}}(U) U$. Below we discuss a wide generalization of the above construction.

The following formulation allows even wider generalizations. Let $\h\to f_{\h}(\partial \Delta_0)\in H_{n-1}(L(\h))$ be a functor which associates to $\h$ the homotopy class in $H_{n-1}(L(\h))$ which is the image of the fundamental class of $\partial \Delta_0$ under the map $f_\h$. That functor has a generalization to the case when instead of the union of hyperplanes one consider the union of affine spaces.

\section{Generalized virtual polyhedra theory}

We will generalized the above construction in the following directions:
\begin{enumerate}
\item  Instead of the union $L(\h)$ of hyperplanes parallel to the faces of a convex polyhedron $\Delta$ we will consider   union $X$ of arbitrary affine subspaces of any dimensions in an affine space;

\item  instead of the image of $\partial \Delta_0$ in $L(\h)$ we will consider arbitrary cycles in $X$.
We will identify homology groups of  unions $X_1$ and $X_2$
of different collections of affine subspaces under some combinatorial  assumptions.

\item  For the case when affine subspaces  are hyperplanes in $\mathbb R^n$  and a cycle having dimension $(n-1)$ the above generalization can be modified. In that modification instead of $(n-1)$-dimensional cycles in the union of hyperplanes  in $\mathbb R^n$ one can deal with $n$-dimensional chains in $\mathbb R^n$ whose boundaries are the above cycles. In a particular case of simplified analogous virtual polyhedra, such chains coincide with chains $\sum W_{f_{\h}}(U)U$ which we discussed above.
\end{enumerate}

In this section we deal with ordered sets of affine subspaces $L_i$ indexed by the same set $I$.

\begin{definition} The set $X=\bigcup_{i\in I} L_i$ has the natural covering by spaces $L_i$. The nerve $K_X$ of the natural covering of $X$ is the following simplicial complex:
\begin{enumerate}
\item  the set of vertices of $K_X$ is the set $I$ of indexes $i$;

\item the set $J\subset I$ of vertices belongs to one simplex if and only if
$
\bigcap_{i\in J}L_i\neq\emptyset$.
\end{enumerate}
\end{definition}

\begin{definition} Let $X_1=\bigcup L_i$ and $X_2=\bigcup M_i$ be unions of affine subspaces in the spaces $L$ and $M$
 indexed by the same set $\{i\}=I$. We will say that:
\begin{enumerate}
\item $X_1$ dominate $X_2$ if the nerve $K_{X_1}$ is a subcomplex of the nerve $K_{X_2}$;

\item $X_1$ and $X_2$ are equivalent if $K_{X_1}=K_{X_2}$.
\end{enumerate}
\end{definition}

\begin{figure}[htbp]
\begin{center}
\includegraphics[scale=0.2]{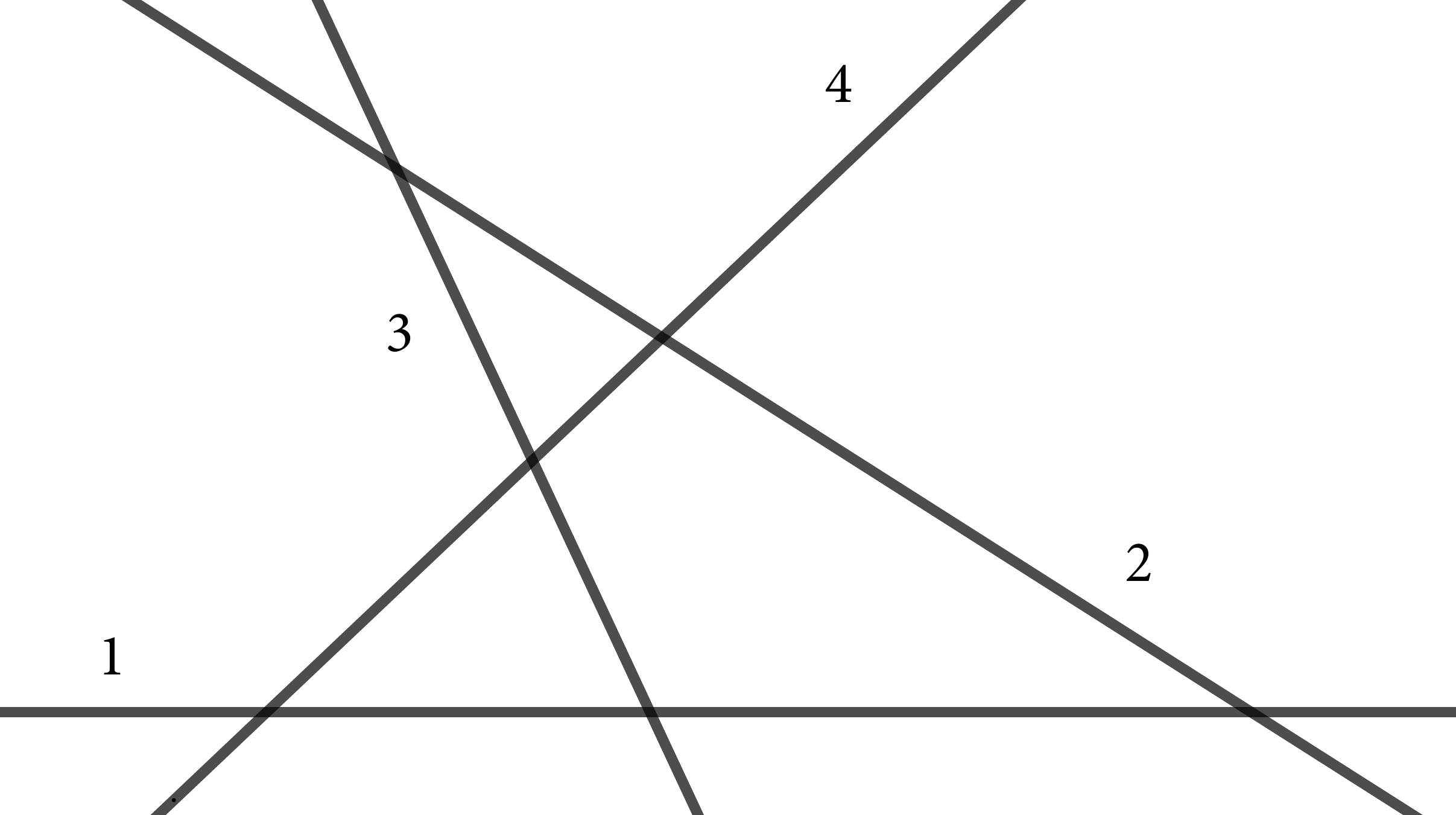}
\end{center}
\caption{An ordered set of four lines on a plane }
\label{fig:fig 3}
\end{figure}

Let $BK_X$ be the barycetric subdivision of the nerve $K_X$. For each $i\in I$ let $B_iK_X$ be the union of all (closed) simplices in $BK_X$ which contain the vertex $A_i$ (corresponding to the space $L_i$ in the covering of $X$).

\begin{lemma} The nerve of covering of $BK_X$ by the closed sets $B_iK_X$ coincides with the original nerve $K_X$.
\end{lemma}

\begin{definition} A map $g:K_{X_1}\to X_2$ is compatible with coverings if for any $i\in I$ and $x\in BX_i$ the image  $g(x)$ belongs to $M_i$.
\end{definition}

\begin{theorem}\label{thm6.1}
\begin{enumerate}
\item A map $g:K_{X_1}\to X_2$ compatible with coverings exists if and only if the inclusion $K_{X_1}\subset K_{X_2}$ holds.

\item All maps from $K_{X_1}$ to $X_2$ compatible with coverings are homotopy equivalent to each other.

\item If $K_{X_1}=K_{X_2}$ then a map $g:K_{X_1}\to X_2$ provides a homotopy equivalence between these spaces.
    \end{enumerate}
\end{theorem}

Theorem~\ref{thm6.1} implies that all cycles of $H_*(X)$ could be seen in homology groups of the nerve $K_X$ of the covering of $X$. Moreover if $K_{X_1}$ is a subcomplex of $K_{X_2}$ then each cycle in $H_*(K_{X_1})$ has natural image in $H_*(X_2)$.

Consider a collection of affine  $k$-dimensional subspaces $\{L_i\}$ in a vector space $L$ with $i\in I$. For each $i$ denote by $Y_i$ the factor space $L/\tilde L_i$ where $\tilde L_i$ is the vector subspace parallel to $L_i$.

\begin{definition} A collection of vectors $y_i\in Y_i$ are compatible with the nerve of $X=\bigcup L_i$ if the following condition holds:
if $L_{i_1}\cap \dots \cap L_{i_m}\neq \emptyset $ then $(L_{i_1}+y_{i_1})\cap \dots\cap (L_{i_m}+y_{i_m})\neq \emptyset$.
\end{definition}

Let $Y$ be the space of all $I$-tuples $y_1,\dots ,y_{|I|}$ compatible with the nerve of $X=\bigcup L_i$.

\begin{definition} To each point ${\bf y}\in Y$ one associates a collection $\{L_i(y)\}$ where $L_i(y)= L_i+y_i$.
\end{definition}

The set $Y$ parametrizes translates of the subspaces $L_i$ which preserve existed intersections. More precisely, for a generic point ${\bf y}\in Y$ the collections $\{L_i({\bf y})\}$ has the same nerve, which we denote by $K_X$. There is a subset $\Sigma$ in $Y$ of a smaller dimension than $Y$, such that the nerve of $\bigcap L_i({\bf y})$ contains $K_X$ as a proper subcomplex.

For any point ${\bf y}\in Y$ one can defined a map $g_{\bf y}:K_X\to \bigcup L_i({\bf y})$ compatible with the nerves of these spaces and depends of $y$ linearly, i.e.
$g_{\lambda {\bf y}_1+\mu {\bf y}_2}=\lambda g_{{\bf y}_1}+\mu g_{{\bf y}_2}$. For any $k$-form $\alpha$ on $L\times Y$ with polynomial coefficients and for any cycle $\gamma\in H_k(K_X)$ one can consider the following function
$F_{\alpha, \gamma}$ on $Y$:
$$F_{\alpha, \gamma}({\bf y}) =  \int_{\gamma} g_{\bf y}^*\alpha.$$

\begin{theorem} The function $F_{\alpha, \gamma}$  is a polynomial function on $Y$.
\end{theorem}

\section{Homotopy type of the union of affine subspaces}

 We know that the homotopy type of $X=\bigcup L_i$ is the same as the homotopy type of its nerve $K_X$.

 For any finite simplicial complex it is easy to construct a collection of affine subspaces whose nerve is homeomorphic to the given complex. However, if affine subspaces have codimension one in $L$ then there union always has a homotopy type of a wedge of spheres.

Let $\{L_i\}$ be a collection of hyperplanes in $L$. Denote by $l(\{L_i\})$ the biggest subspace parallel to all these hyperplanes. One can check that $l(\{L_i\})$ is equal to the intersection of linear subspaces parallel to the affine spaces $L_i$. For sufficiently general collection of hyperplanes the space $l(\{L_i\})$ is equal to zero.

\begin{theorem}\label{thm7.1} The union $X=\bigcup L_i$ of the affine hyperplanes under conditions $l(\{L_i\})=0$ is homotopy equivalent to the wedge of $(n-1)$-dimensional spheres, which are in one to one correspondence with the convex polyhedra which are boundaries of connected  bounded components of $L\setminus \bigcup L_i$.
 \end{theorem}

\begin{corollary}\label{cor7.1} If $l(\{L_i\})$ has dimension $m$, then $X=\bigcup L_i$ has a homotopy type of a wedge of spheres of dimension $n-1-m$.
\end{corollary}
\begin{proof}
Indeed, under conditions of Corollary~\ref{cor7.1} $X$ is equal to the product of $X\cap l^\bot \times l(\{L_i\})$ where $l^\bot $ is a space transversal to $l(\{L_i\})$. For the intersection $X\cap l^\bot$ the above theorem is applicable.
\end{proof}

In the assumptions of Theorem~\ref{thm7.1} let us choose a cycle $\Gamma\in H_{n-1} (K_X,\mathbb Z)$. For a point ${\bf y}\in Y$ consider the map $g_{\bf y}:\Gamma\to \bigcup L_i({\bf y})=L({\bf y})$ compatible with coverings.  To this map one can associate the chain $\sum W_{\tau} (U) U$ where $U$ is a connected component of $L\setminus L({\bf y})$ and $W_\tau (U)$ is the winding number of the map $\tau:\Gamma\to U$ where $\tau =\frac{g_{\bf y}-a}{|g_{\bf y}-a|}$ and $a$ is a point in $U$. This chain can be considered as the generalized virtual polyhedron which appeared in the assumption of Theorems~\ref{thm5.1} and~\ref{thm5.2}. In particular, the integral of a form $\omega=Pdx_1\wedge\dots\wedge dx_n$, where $P$ is a polynomial, against such chains depends polynomially on ${\bf y}$.

\section{Applications of generalized virtual polyhedra}
We finish the paper with a brief description of recent applications of generalized virtual polyhedra. Exact statements and details can be found in \cite{7,8}.

Torus manifolds (see \cite{9,10}) provides a wide topological generalization of smooth algebraic toric varieties. To each such manifold  one can associate the union $L(y)$ of hyperplanes in $\mathbb R^n$ depending on parameters $y$ and a $(n-1)$-dimensional cycle $\Gamma$ in the nerve of the natural covering of $L(y)$ \cite{11,12,7}. Using results which are described  in the previous section one can defined a homogeneous degree $n$ polynomial in $y$ which is the volume of corresponding virtual polyhedra. One can describe the cohomology ring of the torus manifold using Khovanskii--Pukhlikov construction, known in toric varieties theory \cite{5}.

On torus manifold there is a special collection of characteristic linear bundles. Such  bundles are in one-to-one correspondence with generalized virtual polyhedra, responsible for the   torus manifold. Intersection number of $n$-sections of such bundles is equal to $n!$ multiplied by mixed volume of the corresponding virtual polyhedra. This theorem generalize BKK-theorem for   torus manifolds. Moreover one can describe the cohomology ring of a bundle whose fibers are torus manifolds in terms of integrals of some polynomials over corresponding virtual polyhedra (see \cite{8}). This theorem generalized the analogous result for bundles with toric fibers \cite{13}.

\end{document}